\documentclass[jaz]{cupaus}

\usepackage{booktabs}
\usepackage{array,multirow}

\theoremstyle{plain}
\newtheorem{theorem}{Theorem}
\newtheorem{lemma}[theorem]{Lemma}

{\theorembodyfont{\upshape}\newtheorem{definition}[theorem]{Definition}}
{\theorembodyfont{\upshape}\newtheorem{remark}[theorem]{Remark}}
{\theorembodyfont{\upshape}\newtheorem{algorithm}[theorem]{Algorithm}}

\renewcommand{\geq}{\geqslant}
\renewcommand{\leq}{\leqslant}
\renewcommand{\ge}{\geqslant}
\renewcommand{\le}{\leqslant}

\def\GAP{{\sf GAP}}
\def\Magma{{\sc Magma}}

\newcommand{\Alt}{\mathrm{Alt}}
\newcommand{\F}{\mathbb{F}}
\newcommand{\PGL}{\textup{PGL}}
\newcommand{\GL}{\textup{GL}}
\newcommand{\Sym}{\mathrm{Sym}}
\newcommand{\multiset}[1]{\{\kern-1.2pt\{ {#1}\}\kern-1.2pt\}}

\newcommand{\Z}{\mathbb{Z}}

\begin{document}

\hyphenation{}

\verso{S.\,P. Glasby, F. L\"ubeck, A.\,C. Niemeyer and C.\,E. Praeger}
\recto{Primitive prime divisors and the $n$-th cyclotomic polynomial}

\title{Primitive prime divisors and the\\ ${\large \boldsymbol{n}}$-th cyclotomic polynomial}

\cauthormark 

\author[1]{S.\,P. Glasby}
\author[2]{Frank L\"ubeck}
\author[3]{Alice C. Niemeyer}
\author[4]{Cheryl E. Praeger}

\address[1]{Centre for Mathematics of Symmetry and Computation,
  University of Western Australia; also affiliated with The Department of Mathematics, University of Canberra.\; {\tt {\small GlasbyS@gmail.com}}\;
 {\tt\href{http://www.maths.uwa.edu.au/~glasby/}{\small http://www.maths.uwa.edu.au/$\sim$glasby/}}}

\address[2]{Lehrstuhl D f\"ur Mathematik, RWTH Aachen University, Templergraben 64,
  52062 Aachen, Germany.\;{\tt\small  Frank.Luebeck@Math.RWTH-Aachen.De}\newline 
  {\tt\href{http://www.math.rwth-aachen.de/~Frank.Luebeck/}{\small http://www.math.rwth-aachen.de/$\sim$Frank.Luebeck/}}}

\address[3] {Department of Mathematics and Statistics,
NUI Maynooth, Ireland. \;{\tt\small  Alice.Niemeyer@nuim.ie}
}

\address[4]{Centre for Mathematics of Symmetry and
Computation, University of Western Australia;
also affiliated with King Abdulaziz University, Jeddah, Saudi Arabia.\;
{\tt\small  Cheryl.Praeger@uwa.edu.au}\;
{\tt\href{http://www.maths.uwa.edu.au/~praeger}{\small http://www.maths.uwa.edu.au/$\sim$praeger}}}


\pages{1}{13}

\begin{abstract}
Primitive prime divisors play an important role in group theory
and number theory. 
We study a certain number theoretic quantity, called
$\Phi^*_n(q)$,
which is closely related to 
the cyclotomic polynomial $\Phi_n(x)$ and to primitive
prime divisors of $q^n-1$.
Our definition of $\Phi^*_n(q)$ is novel, and we prove it
is equivalent to the definition given by Hering.
Given positive constants $c$ and $k$, we give an algorithm for determining
all pairs $(n,q)$ with $\Phi^*_n(q)\le cn^k$. This algorithm
is used to extend (and correct) a result of Hering which is useful for
classifying certain families of subgroups of finite linear groups.
\end{abstract}

\keywords[2010 \textit{Mathematics subject classification}]{
11T22, 11Y40, 20G05}

\keywords[\textit{Keywords and phrases}]{Zsigmondy primes,
Cyclotomic Polynomial}

\maketitle

\textit{Dedicated to the memory of our esteemed colleague L.G. (Laci) Kov\'acs}

\section{Introduction}
In 1974  Christoph Hering~\cite{Hering} 
classified the subgroups $G$ of the general linear group
$\GL(n,\F_q)$ which act transitively
on the nonzero vectors $(\F_q)^n\setminus\{0\}$. In his
investigations a certain number theoretic function, $\Phi^*_n(q)$, plays an
important role. It divides the $n$th cyclotomic polynomial evaluated at a
prime power~$q$, and hence divides $|(\F_q)^n\setminus\{0\}|=q^n-1$.
It is not hard to prove that $\GL(n,\F_q)$ contains an element of
order~$\Phi^*_n(q)$, and every element $g$ of $\GL(n,\F_q)$ whose order
is not coprime to $\Phi^*_n(q)$ acts irreducibly on the natural module
$(\F_q)^n$, c.f. \cite[Theorem  3.5]{Hering}. 
A key  result \cite[p.1]{Hering} shows that if
$1<\gcd(|G|,\Phi^*_n(q))\le (n+1)(2n+1)$,
then the structure of $G$ is severely constrained.

Our definition  below of~$\Phi_n^*(q)$  differs from  the one  used by
Hering~\cite[p.\;1]{Hering},  L\"uneburg~\cite[Satz~2]{L}  and  Camina
and   Whelan~\cite[Theorem~3.23]{CW},  who used  the   definition   in
Lemma~\ref{Ldef}\eqref{LdefD}.           We           show          in
Section~\ref{sec:equivalentdefs} that our  definition is equivalent to
theirs   and  that $\Phi_n^\ast(q)$  could have  also been
defined in several other ways.

\begin{definition}\label{def:phistar}
Suppose $n,q\in\Z$ are such that $n\ge 1$ and $q\ge2$.  
  Write $\Phi_n(X)$ for the $n$-th cyclotomic polynomial
  $\prod_\zeta (X-\zeta)$ where $\zeta$ ranges over the
  primitive complex $n$-th roots of unity.
  Let $\Phi_n^*(q)$ be the largest
  divisor of $\Phi_n(q)$ which is coprime to 
  $\prod_{1 \leq k < n}(q^k-1)$.
\end{definition}

Our definition of $\Phi_n^*(q)$ is motivated by the numerous
applications of primitive prime divisors, see \cite{NieP}
or \cite{BambergPenttila,GPPS}. As our primary motivation is geometric,
we will assume later (after Section~\ref{M})
that~$q$ is a prime power; before this point $q\ge 2$ is arbitrary
unless otherwise stated.
A divisor $m$ of $q^n-1$ is called a {\it strong primitive divisor of
  $q^n-1$} if $\gcd(m,q^k-1)=1$ for $1 \leq k < n$, and a {\it weak
  primitive divisor of $q^n-1$} if $m\nmid (q^k-1)$ for $1 \leq k < n$.
By our definition,
$\Phi_n^*(q)$ is the largest strong primitive divisor of $q^n-1$.
A primitive divisor of $q^n-1$ which is prime is called a
{\it primitive prime divisor} ({\it ppd}) of $q^n-1$ or a Zsigmondy
prime (``strong'' equals ``weak'' for primes).
DiMuro~\cite{DiMuro} uses 
weak primitive {\it prime power} divisors or {\it pppds}
 to extend the classification
in~\cite{GPPS} to $d/3< n\le d$. Our application in Section~\ref{sec:Ac}
has $d/4\le n\le d$.

Primitive  prime divisors  have  been  studied since  Bang~\cite{Bang}
proved in  1886 that  $q^n-1$ has  a primitive  prime divisor  for all
$q\ge2$ and $n>2$ except for $q=2$ and $n=6$. Given coprime integers $q>r\ge1$ and
$n>2$, Zsigmondy~\cite{Z} proved in 1892 that there exists a prime~$p$
dividing $q^n-r^n$ but not $q^k-r^k$ for $1\le k<n$ except when $q=2$,
$r=1$, and $n=6$.  The Bang-Zsigmondy theorem has been  reproved many times
as explained in~\cite[p.\;27]{Rib} and~\cite[p.\,3]{DHP}; modern proofs
appear in~\cite{L,R}.  Feit~\cite{Feit}
studied `large Zsigmondy primes', and  these play a fundamental role
in  the   recognition  algorithm  in   \cite{NieP}.   Hering's
results in \cite{Hering} influenced  subsequent work on linear groups,
including  the classification  of linear  groups containing  primitive
prime  divisor  (ppd)-elements  \cite{GPPS}, and  its  refinements  in
\cite{ BambergPenttila, DiMuro, NieP}.

We  describe algorithms in Sections~\ref{M} and~\ref{M*}
which,  given positive constants $c$ and  $k$, list
all  pairs $(n,  q)$ for  which $n\ge3$ and $\Phi^*_n(q)\leq  c  n^k$. 
The behaviour of $\Phi_n^*(q)$ for $n=2$ is different from that for
larger~$n$ (see Lemma~\ref{Ldef}(\ref{LdefC}) and Algorithm~\ref{alg:Mstarn2}).

\begin{theorem}\label{the:one} Let $q\ge2$ be a prime power.
\begin{enumerate}[{\rm (a)}]
  \item There is an algorithm which, given constants $c,k>0$ as input, outputs
  all pairs $(n,q)$ with $n\ge3$ and $q\ge2$ a prime power such
  that $\Phi^*_n(q)\le cn^k$.
  \item If $n\ge 3$, then $\Phi^*_n(q)\leq n^4$ if and only if $(n,q)$
    is listed in \emph{Tables~\ref{tab:onea}, \ref{tab:six}} or
    \emph{\ref{tab:oneb}}. Moreover, the prime powers $q$ with
    $q\le 5000$ and $\Phi^*_2(q)\leq 2^4=16$ are listed in
    \emph{Table~\ref{tab:q=2}}.
\end{enumerate}
\end{theorem}

In some group theoretic applications we need explicit information
about~$\Phi_n^*(q)$ when this quantity is considerably larger than~$n^4$,
but we have tight control over the sizes of its ppd divisors (each of
which must be of the form~$in+1$ by
Lemma~\ref{lemma:phiq}(\ref{lemma:phiq1modn})). We give an example of
this kind of result in Theorem~\ref{the:two}, where we require that the ppd divisors
are sufficiently small for our group theoretic application in
Section~\ref{sec:Ac}.
This  motivated  our
effort to  strengthen Hering's  result and  we discovered  two missing
cases in \cite[Theorem 3.9]{Hering}; see Remark~\ref{rem:one}.
We list in Theorem~\ref{the:one}
all pairs $(n,q)$ with $n \geq 3$ and $q\ge2$ a prime power for which
$\Phi^*_n(q)\le n^4$; the implementations in~\cite{G} can handle much
larger cases like $\Phi^*_n(q)\le n^{20}$.
In Theorem~\ref{the:two} we also require that the ppd 
divisors of $\Phi^*_n(q)$ be small
for our group theoretic application in
Section~\ref{sec:Ac}. 

\begin{theorem}\label{the:two}
Suppose   that  $q\ge2$   is  a  prime  power and   $n\geq 3$.  
Then all  possible values of $(n,  q)$  such that 
$\Phi^*_n(q)$ has a prime factorisation  of
the form   $\prod_{i=1}^4  (i  n+1)^{m_i},$  with   $0\le m_1\le  3$ and
$0\le m_2,m_3,m_4\le 1$
are listed  in \emph{Table~\ref{tab:two}}.
\end{theorem}
 
The proof of Theorem~\ref{the:one}(a) rests on 
the correctness of Algorithms~\ref{alg:M} and~\ref{alg:Mstarnbig} which
are proved in Sections~\ref{M} and~\ref{M*}. 
Theorem~\ref{the:one}(b) and \ref{the:two} follow by applying these algorithms.
For Theorem~\ref{the:two} we  observe that 
$\Phi_n^\ast(q) \le (n+1)^3\prod_{i=2}^4  (i  n+1) \le  16n^7$ for all $n\ge 4$,
whereas  for $n=3$ only $2n+1$ and $4n+1$ are primes and again
$\Phi_n^\ast(q) \le 7\cdot13 \le 16n^7.$
Thus the entries in Table~\ref{tab:two} were obtained by 
searching the output of our algorithms to find
the pairs $(n,q)$ for which $\Phi^*_n(q)\le 16n^7$  and has the given
factorisation. This factorisation arose from the application
(Theorem~\ref{the:Ac}) in Section~\ref{sec:Ac}.

\begin{remark}\label{rem:one}
The missing cases in part~(d) of \cite[Theorem 3.9]{Hering} had
$\Phi^*_n(q)=(n+1)^2$. We discovered the possibilities
$n=2, q=17,$ and $n=2, q=71$ when comparing Hering's result with output of the
\Magma~\cite{BCP} and \GAP~\cite{GAP} implementations of our algorithms,
see Table~\ref{tab:q=2}.
\end{remark}

\section{Cyclotomic polynomials: elementary facts}

The product $\prod_{1 \leq k < n}(q^k-1)$ has no factors when $n=1$.
An empty product is~1, by convention, and so $\Phi_1^*(q)=\Phi_1(q)=q-1$. 

The M\"obius function $\mu$ satisfies $\mu(n)=(-1)^k$ if
$n=p_1\cdots p_k$ is a product of distinct primes, and $\mu(n)=0$
otherwise. Our algorithm  uses the following elementary facts.

\begin{lemma}\label{lemma:phiq}
Let $n$ and $q$ be integers satisfying $n\ge1$ and $q\ge2$.
\begin{enumerate}[{\rm (a)}]
  \item\label{lemma:phiqmoebius}
  The polynomial $\Phi_n(X)$ lies in $\Z[X]$ and is irreducible. Moreover,
  \[
    X^n-1 = \prod_{d|n} \Phi_d(X)\qquad\textrm{and}\qquad
    \Phi_n(X) = \prod_{d|n} (X^\frac{n}{d}-1)^{\mu(d)}.
  \]
  \item\label{lemma:phidivide} If $d\mid n$ and $d>1$, then $\Phi_n(X)$ divides
  $(X^n-1)/(X^{n/d}-1)=\sum_{i=0}^{d-1} (X^{n/d})^i$.
  \item\label{lemma:phiq1modn}
  If $r$ is a prime and $r\mid \Phi_n^*(q)$, then $n$ divides $r-1$,
  equivalently $r \equiv 1\pmod{n}$.
  \item\label{lemma:phiqincreasing}
  For any fixed integer $n\ge1$ the function $\Phi_n(q)$ is strictly increasing
  for $q>1$.
  \item\label{lemma:phiqeuler} 
  Let $\varphi$ be Euler's totient function which satisfies
  $\varphi(n) = \deg(\Phi_n(X))$. Then 
  \[
    \varphi(n) \geq \frac{n}{\log_2(n)+1}\qquad\textup{for $n\ge1$.}
  \]
  \item\label{lemma:phiqbound}
  For all $n\ge2$ and $q\ge2$ we have
  $q^{\varphi(n)}/4< \Phi_n(q)<4q^{\varphi(n)}$.
\end{enumerate}
\end{lemma}

\begin{proof}
\noindent(\ref{lemma:phiqmoebius})
The irreducibility of $\Phi_n(X)\in\Z[X]$ and the other facts, are proved
in~\cite[\S13.4]{DF}.

\noindent(\ref{lemma:phidivide})
By part~\eqref{lemma:phiqmoebius} $(X^n-1)/(X^{n/d}-1)$ equals
$\prod_k\Phi_k(X)$ where $k\mid n$ and $k\nmid (n/d)$. Since
$d>1$, it follows that $\Phi_n(X)$ is a factor in this product.

\noindent(\ref{lemma:phiq1modn}) 
If $r\mid\Phi_n^*(q)$ then $r\mid(q^n-1)$ and $n$ is the order of $q$ modulo $r$,
so $n\mid(r-1)$.

\noindent(\ref{lemma:phiqincreasing})
This follows from Definition~\ref{def:phistar}
because $\Phi_n(q)=|\Phi_n(q)|=\prod_\zeta|q-\zeta|$ and $|\zeta|=1$.

\noindent(\ref{lemma:phiqeuler})
We use the formula $\varphi(n) = n \prod_{i=1}^t \frac{p_i-1}{p_i}$ where
$p_1 < p_2 < \cdots < p_t$ are the prime divisors of $n$. Using the trivial
estimate $p_i \geq i+1$ we get $\varphi(n) \geq n/(t+1)$. It follows from
$2^t\le p_1p_2\cdots p_t\le n$ that $t \leq \log_2(n)$. Hence
$\varphi(n) \geq n/(\log_2(n)+1)$ as claimed.

\noindent(\ref{lemma:phiqbound})
Using the product formula for $\Phi_n(X)$ in~(\ref{lemma:phiqmoebius}) and
$\mu(d) \in \{0,-1,1\}$, we see that $\Phi_n(q)$ equals $q^{\varphi(n)}$
times a product of distinct factors of the form $(1-1/q^i)^{\pm 1}$ with $1 \leq
i \leq n$.  Since $\prod_{i=1}^\infty (1-1/q^i) \geq \prod_{i=1}^\infty
(1-1/2^i) = 0.28878\cdots > 1/4$ we get 
\[ \frac{q^{\varphi(n)}}{4} < \Phi_n(q) < 4 q^{\varphi(n)}.\]
\end{proof}

\begin{remark} 
Hering~\cite[Theorem~3.6]{Hering} gives sharper estimates than those in
Lemma~\ref{lemma:phiq}(\ref{lemma:phiqbound}).
But our (easily established) estimates suffice for the
efficient algorithms below.
\end{remark}

\section{Equivalent definitions of \texorpdfstring{$\Phi_n^*(q)$}{}}\label{sec:equivalentdefs}

We now state equivalent ways  in which to define $\Phi_n^\ast(q)$ where $q\ge2$
is an integer. Because our motivation for studying $\Phi^*_n(q)$ arose from
finite geometry, we assume after the proof of Lemma~\ref{Ldef}
that~$q$ is a prime power. Observe that
Lemma~\ref{Ldef}\eqref{LdefC} suggests a much faster algorithm
for computing~$\Phi_n^*(q)$ than does Definition~\ref{def:phistar}.

\begin{lemma}\label{Ldef}
Let $n,q$ be integers such that $n\ge2$ and $q\ge2$.
The following statements could be used as alternatives to the definition
of~$\Phi^*_n(q)$ given in Definition~\ref{def:phistar}.
\begin{enumerate}[{\rm (a)}]
  \item\label{LdefA}
   $\Phi_n^*(q)$ is the largest divisor
   of $\Phi_n(q)$ coprime to $\prod_{k\mid n,\,k<n}\Phi_k(q)$.
  \item\label{LdefC}
  Let $(q+1)_2$ be the largest power of~$2$ dividing $q+1$, and let~$r$ be
  the largest prime divisor of $n$. Then
  \[
   \Phi_n^*(q) = \begin{cases} (q+1)/(q+1)_2\quad\quad&\textup{if $n=2$,}\\
   \Phi_n(q)&\textup{if $n>2$ and $r\nmid\Phi_n(q)$,}\\
   \Phi_n(q)/r &\textup{if $n>2$ and $r\mid\Phi_n(q)$.}\end{cases}
  \]
  \item\label{LdefD} $\Phi_n^*(q)=\Phi_n(q)/f^i$ where $f^i$ is the largest
  power of $f:=\gcd(\Phi_n(q),n)$ dividing $\Phi_n(q)$. 
\end{enumerate}
\end{lemma}

\begin{remark}\label{Rem}
For $n>2$ the last paragraph of the proof of part~\eqref{LdefC}
shows that $d:=\gcd(\Phi_n(q),\prod_{1\le k<n}(q^i-1))$ equals
$f:=\gcd(\Phi_n(q),n)$. Either $d=f=1$ and $r\nmid\Phi_n(q)$,
of $d=f=r$ and $r\mid\Phi_n(q)$. Thus, part~\eqref{LdefD} simplifies to
$\Phi_n^*(q)=\Phi_n(q)/f$ when $n>2$.
\end{remark}

\begin{proof}
\noindent\eqref{LdefA}
We use the following notation where $m$ is a divisor of $\Phi_n(q)$:
\begin{align*}
  P_n&=\prod_{1\le k<n}(q^k-1), &P_n^{\,\prime}&=\prod_{k\mid n,\,k<n}\Phi_k(q),\\
  d_n(m)&=\gcd(m,P_n), \quad &d^{\,\prime}_n(m)&=\gcd(m,P_n^{\,\prime}).
\end{align*}

Fix a divisor $m$ of $\Phi_n(q)$. We prove that $d_n(m)=1$ holds if and
only if $d^{\,\prime}_n(m)=1$. Certainly $d_n(m)=1$ implies
$d^{\,\prime}_n(m)=1$ as $P_n^{\,\prime}\mid P_n$. Conversely, suppose that $d_n(m)\ne1$.
Then there exists a prime divisor $r$ of $m$
that divides $q^k-1$ for some $k$ with $1\le k<n$. However,
$r\mid\Phi_n(q)\mid (q^n-1)$ and $\gcd(q^n-1,q^k-1)=q^{\gcd(n,k)}-1$, so
$r$ divides $q^{\gcd(n,k)}-1$. Hence $r$ divides $\Phi_\ell(q)$ for some
$\ell\mid\gcd(n,k)$ by Lemma~\ref{lemma:phiq}(\ref{lemma:phiqmoebius}).
In summary, $r\mid d_n(m)$ implies $r\mid d^{\,\prime}_n(m)$,
so $d_n(m)\ne1$ implies $d^{\,\prime}_n(m)\ne1$.

For any divisor~$m$ of $\Phi_n(q)$ we have shown that
$\gcd(m,P_n)=1$ holds if and only if $\gcd(m,P_n^{\,\prime})=1$.
Thus the largest divisor of $\Phi_n(q)$ coprime to $P_n^{\,\prime}$
is equal to the largest such divisor which is coprime to $P_n$,
and this is $\Phi^*_n(q)$ by
Definition~\ref{def:phistar}.


\noindent\eqref{LdefC}
First consider the case $n=2$. Now $d:=d_2(\Phi_2(q))=\gcd(q+1, q-1)$
divides~2. Indeed, $d=1$
for even $q$, and $d=2$ for odd $q$. In both cases, $(q+1)/(q+1)_2$ is the
largest divisor of $q+1$ coprime to $q-1$. Thus $\Phi^*_2(q)=(q+1)/(q+1)_2$
by Definition~\ref{def:phistar}.

Assume now that $n>2$. Let $d=\gcd(\Phi_n(q),P_n)$ where
$P_n=\prod_{1\le k<n}(q^k-1)$. If $d=1$, then
$\Phi_n^*(q) = \Phi_n(q)$ by Definition~\ref{def:phistar}. Suppose that
$d>1$ and~$p$ is a prime divisor of~$d$. Then the order of $q$ modulo~$p$ is
less than $n$, and Feit~\cite{Feit} calls~$p$ a non-Zsigmondy prime.
It follows from~\cite[Proposition~2]{R} or L\"uneburg~\cite[Satz~1]{L}
that the prime $p$ divides $\Phi_n(q)$ exactly once, and $p=r$ is the largest
prime divisor of $n$. Thus we see that $\gcd(\Phi_n(q)/r,P_n)=1$
and $\Phi_n^*(q) = \Phi_n(q)/r$ by Definition~\ref{def:phistar}.
This proves~\eqref{LdefC}.

To connect with part~\eqref{LdefD}, we prove when $n>2$ that $d$ equals
$f:=\gcd(\Phi_n(q),n)$.
Indeed, we prove Remark~\ref{Rem} that either $d=f=1$ and $r\nmid\Phi_n(q)$,
or $d=f=r$ and $r\mid\Phi_n(q)$.
If $d=1$, then $\Phi_n^*(q) = \Phi_n(q)$ and a prime divisor~$p$
of $\Phi_n^*(q)$ satisfies $p\equiv1\pmod n$ by
Lemma~\ref{lemma:phiq}(\ref{lemma:phiq1modn}) and hence
$p\nmid n$. Thus $f=1$ and $r\nmid\Phi_n(q)$ since $r\mid n$.
Conversely, suppose that $d>1$. The previous paragraph shows that
$d=r$ and $r^2\not\mid\Phi_n(q)$. Thus $r\mid f$.
Let $p$ be a prime dividing $f=\gcd(\Phi_n(q),n)$.
Since $\Phi_n(q) \mid (q^n-1)$,
we have $p\mid(q^n-1)$, and hence $p\not\mid\Phi^*_n(q)$ by
Lemma~\ref{lemma:phiq}(\ref{lemma:phiq1modn}).
Thus $p$ divides $P_n$ by Definition~\ref{def:phistar}, and
hence $p$ divides $d=\gcd(\Phi_n(q),P_n)$. However, $d=r$ and so
$p=r=f$, and in this case $r\mid\Phi_n(q)$.  

\eqref{LdefD} By part~\eqref{LdefC} and the last paragraph of the
proof of~\eqref{LdefC}, Definition~\ref{def:phistar} is
equivalent to Hering's definition~\cite{Hering} in part~\eqref{LdefD}.
\end{proof}

\begin{remark}\label{REM}
When~$q$ is a prime power, there is a fourth equivalent definition:
\newline
$\Phi_n^*(q)$ is the order of the largest subgroup of
$\F_{q^n}^\times$ (the multiplicative group of $q^n-1$ nonzero elements
of $\F_{q^n}$) that intersects trivially all the subgroups
$\F_{q^d}^\times$ for $d\mid n$, $d<n$. 
\begin{proof}
The correspondence $H\leftrightarrow |H|$ is a bijection between the
subgroups $H$ of the cyclic group~$\F_{q^n}^\times$ and the divisors of
$q^n-1$. Suppose $d\mid n$. Note that $H\cap \F_{q^d}^\times=\{1\}$
holds if and only if $\gcd(|H|,q^d-1)=1$ as $\F_{q^n}^\times$ is cyclic.
Thus there exists a unique subgroup $H$ whose order $m$ is maximal
subject to $H\cap \F_{q^d}^\times=\{1\}$ for all $d\mid n$, $d<n$.
Hence $m$ is the largest divisor of $q^n-1$ satisfying $\gcd(m,q^d-1)=1$
for all $d\mid n$, $d<n$.
Since $q^n-1=\prod_{d\mid n}\Phi_d(q)$ and $\Phi_d(q)\mid q^d-1$, we see that
$m\mid \Phi_n(q)$. It follows from Lemma~\ref{Ldef}~\eqref{LdefA} that $\Phi^*_n(q)=m$.
\end{proof}
\end{remark}

\section{The polynomial bound \texorpdfstring{$\Phi_n(q)\le c n^k$}{}}\label{M}

As we will discuss in Section~\ref{M*}, the number of pairs $(2,q)$ with $q$
a prime power satisfying $\Phi_2(q)\leq c 2^k$ is potentially infinite.
We therefore deal here with pairs $(n,q)$ for $n\ge3$. 
Given positive constants $c$ and $k$, we now
describe an algorithm for determining all pairs in the set
\[
  M(c,k) := \{(n,q)\in \Z\times\Z\mid n\ge3, q\ge2\textup { a prime power, and }\Phi_n(q) \leq c n^k\}.
\]

\begin{algorithm}{\texorpdfstring{$M(c,k)$}{}}\label{alg:M}
\vskip1.5mm
\noindent{\bf Input:} Positive constants $c$ and $k$.\newline
\noindent{\bf Output:} The finite set $M(c,k)$.\newline

\def\Def#1{[#1]}

\begin{enumerate}[{\bf {\bf\ref{alg:M}.}1}]
\item \Def{Definitions} Set $s := 2+\log_2(c)$, $t := (s+k)/\ln(2)$,
$u := k/\ln(2)^2$ and $b := e^{1-t/(2u)}$ and define for $x\geq 3$
the function $g(x) := x-s-t\ln(x)-u\ln(x)^2$ where $\ln(x)=\log_e(x)$.
Note that $g(x)$ has derivative $g'(x) := 1-t/x-2u\ln(x)/x$.
\item \Def{Initialise} Set $n:=3$ and set $M(c,k)$ to be the empty set.
\item \Def{Termination criterion} If $n > b$ and  $g(n)>0$ and $g'(n)>0$ then return $M(c,k)$.
\item \Def{For fixed $n$, find all $q$}  If $g(n) < 0$ and $2^{\varphi(n)-2} < cn^k$ then compute
$\Phi_n(X)$ and find the smallest
prime power $\tilde q$ such that $\Phi_n(\tilde q) > cn^k$;
add $(n,q)$ to $M(c,k)$ for all prime powers $q<\tilde q$.
\item \Def{Increment and loop} Set $n := n+1$ and go back to step~\ref{alg:M}.3.
\end{enumerate}
\end{algorithm}

\proc{Proof of correctness.}
Algorithm~\ref{alg:M} starts with $n=3$ and it continues to increment $n$.
We must prove that it does terminate at step~\ref{alg:M}.3, and that it
correctly returns $M(c,k)$. Note first that for fixed $n$
the values  $\Phi_n(q)$ are strictly increasing with $q$ by
Lemma~\ref{lemma:phiq}(\ref{lemma:phiqincreasing}).
Thus it follows from Lemma~\ref{lemma:phiq}\eqref{lemma:phiqeuler}
and \eqref{lemma:phiqbound} that
\[
  \Phi_n(q)\ge\Phi_n(2) > \frac{2^{\varphi(n)}}{4}=2^{\varphi(n)-2}
  \geq 2^{n/(\log_2(n)+1) - 2}.
\]
Consider the inequality $2^{n/(\log_2(n)+1) - 2} \ge c n^k$. Taking base-2
logarithms shows
\begin{align*}
  n&\ge(k\log_2(n)+\log_2(c)+2)(\log_2(n)+1)\\
   &=(\log_2(c)+2)+(k+\log_2(c)+2)\log_2(n)+k\log_2(n)^2\\
   &=s+t\ln(n)+u\ln(n)^2
\end{align*}
where the last step uses $\log_2(n)=\ln(n)/\ln(2)$ and the definitions in 
step~\ref{alg:M}.1.
In summary, $2^{n/(\log_2(n)+1) - 2} \ge c n^k$
is equivalent to $g(n) \ge 0$ with $g(n)$ as defined in step~\ref{alg:M}.1.

The inequalities above show that the conditions $g(n) < 0$ and
$2^{\varphi(n)-2} < cn^k$, which we test in step~\ref{alg:M}.4, are necessary for
$\Phi_n(2) \leq c n^k$. We noted above that for fixed $n$ the values
of $\Phi_n(q)$ strictly increase with~$q$.
Thus (if executed for a particular $n$) step~\ref{alg:M}.4 correctly adds to $M(c,k)$ all pairs $(n,q)$
for prime powers~$q$ such that $\Phi_n(q) \le cn^k$.

It remains to show (i) that the algorithm terminates, and (ii) that 
the returned set $M(c,k)$ contains \emph{all} pairs $(n,q)$ such that 
 $\Phi_n(q) \le cn^k$. 
The second derivative of $g(x)$ equals $g''(x) = (t- 2u(1-\ln(x)))/x^2$.
Since $u>0$ this shows that $g''(x) > 0$ if and only if $x > b = e^{1-t/(2u)}$.
Thus $g'(x)$ is increasing for all $x>b$.
Because $x$ grows faster than any power of $\ln(x)$ we have that
$g(x) > 0$ and $g'(x) > 0$ for~$x$ sufficiently large. Thus there exists a (smallest) integer
$\tilde n$ fulfilling the conditions in step~\ref{alg:M}.3, that is,
$\tilde n > b$, $g(\tilde n)>0$ and $g'(\tilde n)>0$. The algorithm terminates when
step~\ref{alg:M}.3 is executed for the integer $\tilde n$.
To prove that the returned set $M(c,k)$ is complete, we verify that,
for all $n\ge\tilde n$, there is no prime power $q$ such that $\Phi_n(q)\le cn^k$. 
Now, for all $x \geq \tilde n$, we have
$x >b$ so that $g'(x)$ is increasing for $x \geq
\tilde n$, and so $g'(x)\ge g'(\tilde n)>0$, whence
$g(x)$ is increasing for $x \geq \tilde n$. In particular,
$n \geq \tilde n$ implies that $g(n)\ge g(\tilde n) > 0$ and so (from our displayed computation above),
for all prime powers $q$, $\Phi_n(q)\ge\Phi_n(2) > c n^k$. Thus there are no pairs $(n,q)\in M(c,k)$
with $n\ge\tilde n$, so the returned set $M(c,k)$ is complete.
\ep\mb

\section{Determining when \texorpdfstring{$\Phi_n^*(q)\le c n^k$}{}}
\label{M*}

We describe an algorithm to determine all pairs $(n, q)$, with $n,q\ge2$ and
$q$ a prime power, such that the value
$\Phi_n^*(q)$ is bounded by a given polynomial in $n$, say $f(n)$.
For $n\ge3$ the algorithm
determines the finite list of possible $(n,q)$. For $n=2$ the output is split
between a finite list which we determine, and a potentially infinite (but very
restrictive) set of prime powers $q$ of the form $2^am-1$ where $m\le f(2)$
is odd.
Table~\ref{tab:q=2} lists the prime powers $q\leq 5000$ such that
$\Phi_2^*(q)\le 16$;  we see that some proper powers occur, though
the majority of the entries are primes. 
For example, if  $\Phi_2^*(q) =1$ then
the prime powers $q$ of the form $2^a-1$, must be a prime by~\cite{Z}. Such
primes are called Mersenne primes.

The set $M(c,k)$ of all pairs $(n,q)$ satisfying $\Phi_n(q)\le cn^k$
is finite by Lemma~\ref{lemma:phiq}(\ref{lemma:phiqbound}). By
contrast the set of pairs $(n,q)$ satisfying $\Phi^*_n(q)\le cn^k$ may
be infinite as $\Phi^*_2(q)=m$, $m$ odd, may have infinitely many (but
highly restricted) solutions for $q$.
Algorithm~\ref{alg:Mstarnbig} computes the following set (which we see
below is a finite set)
\[
  M^*_{\ge3}(c,k) = \left\{ (n,q)\in \Z\times\Z\mid n\ge3, q\ge2\textup { a prime power, and }\Phi^*_n(q) \leq c n^k\right\}.
\]

\begin{algorithm}{\texorpdfstring{$M^*_{\ge3}(c,k)$}{}}\label{alg:Mstarnbig}
\vskip1.5mm
\noindent{\bf Input:} Positive constants $c$ and $k$.\newline
\noindent{\bf Output:} The finite set $M^*_{\ge3}(c,k)$.\newline

\begin{enumerate}[{\bf {\bf\ref{alg:Mstarnbig}.}1}]
  \item 
  Compute $M(c,k+1)$ with Algorithm~$\ref{alg:M}$. 
  \item 
  Initialise $M^*_{\ge3}(c,k)$ as the empty set. For all $(n,q)\in 
  M(c,k+1)$ with $n \geq 3$ check if $\Phi^*_n(q) \leq c n^k$. If yes,
  add $(n,q)$ to $M^*_{\ge3}(c,k)$.
  \item 
  Return $M^*_{\ge3}(c,k)$.
\end{enumerate}
\end{algorithm}

\proc{Proof of correctness.}
We need to show that all $M^*_{\ge3}(c,k) \subseteq M(c,k+1)$. This
follows from Lemma~\ref{Ldef}\eqref{LdefC} which shows that
$n \Phi^*_n(q) \geq \Phi_n(q)$ whenever $n \geq 3$.
\ep\mb

\noindent\textsc{Case $n=2$.} We treat the case $n=2$ separately as the
classification has a finite part and a potentially infinite part.
Suppose $q$ is odd and $\Phi_2^*(q)=\frac{q+1}{2^a}=m\le cn^k$
where $m$ is odd by Lemma~\ref{Ldef}\eqref{LdefC}. Then solving for~$q$
gives $q=2^am-1$.

If $m=1$ then $q=2^a-1$ is a (Mersenne) prime as remarked in the first
paragraph of this section. 
Lenstra-Pomerance-Wagstaff conjectured~\cite{LPW} that there are
infinitely many   Mersenne primes, and the asymptotic density of the set
$\{a<x\mid 2^a-1\textrm{ prime}\}$ is ${\rm O}(\log x)$. For fixed $m$
with $m>1$, the number of prime powers of the form $2^am-1$ may also be
infinite (although in this case we cannot conclude that $a$ must be prime).
The set
\[
  M_2^*(c,k)=\{(2,q)\mid  \Phi^*_n(q) \leq c 2^k \textup{ and $q$ is a prime power} \}
\]
is a disjoint union of three subsets:
\begin{align*}
  R(c,k)&:=\{(2,q)\mid (2,q) \in M_2^*(c,k)\text{ and } q\not\equiv{3}\!\!\!\!\pmod4 \},\\
  S(c,k)&:=\{(2,q)\mid (2,q) \in M_2^*(c,k)\text{ and }q\equiv 3\!\!\!\pmod4 \text{ and } q \text{ not prime }  \},\\
  T(c,k)&:=\{(2,q)\mid (2,q) \in M_2^*(c,k)\text{ and } q \equiv 3\!\!\!\pmod4\textup{
    and $q$ prime}\}\\
\end{align*}
As the set  $T(c,k)$ may be infinite
Algorithm~\ref{alg:Mstarn2} below takes as input a constant $B>0$ and
computes  the finite subset $T(c,k,B) =\{ (2,q)\mid q \in  T(c,k)\textup{ and } q\le B\}$ of $M_2^\ast(c,k).$  Table~\ref{tab:q=2} has $n=2$ and $q\le
5000$, so we input $B=5000$.

\begin{algorithm}{\texorpdfstring{$M^*_2(c,k,B)$}{}}\label{alg:Mstarn2}
\vskip2mm
\noindent{\bf Input:} Positive constants $c, k$ and $B$.\newline 
\noindent{\bf Output:} The (finite) set $R(c,k)\cup S(c,k)\cup T(c,k,B)$, see the notation above.\newline

\begin{enumerate}[{\bf {\bf\ref{alg:Mstarn2}.}1}]
  \item 
  Initialise each of $R(c,k), S(c,k), T(c,k,B)$ as the empty set.
  \item 
  Add $(2,q)$ to $R(c,k)$ when $q$ is a power of $2$ with
  $q+1 \leq c 2^k$.
  \item 
  Add $(2,q)$ to $R(c,k)$ when $q$ is a prime power, $q \equiv 1
  \pmod{4}$ and $(q+1)/2 \leq c 2^k$.
  \item For all primes $p\equiv 3 \pmod{4}$ with
  $p \le B$ and $(p+1)/(p+1)_2\le c2^k$ add $(2,p)$ to $T(c,k,B)$.
  For all primes $p\equiv 3 \pmod{4}$ (where $p\le c2^{k-1}$ is allowed) and
  all odd $\ell\ge3$ with $\sum_{i=0}^{\ell-1}(-p)^i\leq c 2^k$
  add   $(2,p^\ell)$ to $S(c,k)$ if $\Phi_2^*(p^\ell)\leq c 2^k$.
  \item Return $R(c,k)\cup S(c,k)\cup T(c,k,B)$.
\end{enumerate}
\end{algorithm}

\proc{Proof of correctness.}
By Lemma~\ref{Ldef}\eqref{LdefC},
$\Phi^*_2(q) = \Phi_2(q) = q+1$ when $q$ is an even prime power and
$\Phi^*_2(q) = \Phi_2(q)/2 = (q+1)/2$ if $q \equiv 1 \pmod{4}$. It is clear that
steps~\ref{alg:Mstarn2}.2 and~\ref{alg:Mstarn2}.3 find 
all pairs $(2,q) \in R(c,k)$ with $q \not\equiv 3 \pmod{4}$, and there
are finitely many choices for $q$.

Any prime power $q\equiv 3 \pmod{4}$ is an odd power $q=p^\ell$ of a prime
$p \equiv 3 \pmod{4}$. Write $q+1 = 2^a m$ with $m$ odd and $a\ge2$,
then $\Phi_2^*(q) = m$.  If $q$ is a prime $(2,q) \in T(c,k,B)$ if and only if 
$q\le B$ and $\Phi_2^*(q) \leq c 2^k$, so step~\ref{alg:Mstarn2}.4 adds such pairs.
This is because, when $q\equiv 3\pmod{4}$ and $q\leq B$ we have,
by Lemma~\ref{Ldef}\eqref{LdefC}, that $\Phi_2^*(q)=(q+1)/2 \leq B$. 
Suppose $q$ is not a prime, that is $\ell>1$. Then we have the 
factorisation $q+1 = (p+1)(\sum_{i=0}^{\ell-1}(-p)^i)$ where
the second factor is odd and so divides $m$. Since
$2p^{\ell-2}\le p^{\ell-2}(p-1)<\sum_{i=0}^{\ell-1}(-p)^i\le m$
and we require $m\le c2^k$, we see $p^{\ell-2}\le c2^{k-1}$.
Since there are finitely many solutions to $p^{\ell-2}\le c2^{k-1}$ with
$\ell>1$ odd, $S(c,k)$ is a finitely set, and step~\ref{alg:Mstarn2}.4
correctly computes $S(c,k)$. Finally, the disjoint union
$R(c,k)\cup S(c,k)\cup T(c,k,B)$ is the desired output set.
\ep\mb

\proc{Proofs of Theorems~\textup{\ref{the:one}}
and~\textup{\ref{the:two}}.}
Theorem~\ref{the:one}(a) follows from the correctness of
Algorithms~\ref{alg:M} and~\ref{alg:Mstarnbig}, and Theorem~\ref{the:one}(b)
uses these algorithms with $(c,k)=(1,4)$. Similarly, Theorem~\ref{the:two}
uses these algorithms with $(c,k)=(16,7)$. It is shown that
in the penultimate paragraph of the proof of Theorem~\ref{the:Ac} that
$\Phi^*_n(q)\le16n^7$ holds for $n\ge 4$.
If $n=3$ and $1\le i\le4$, then $in+1$ is prime for $i=2,4$,
and again $\Phi_n^\ast(q) \le 7\cdot13 \le 16n^7$ holds. We then search
the (rather large) output set for the pairs $(n,q)$ for which
$\Phi^*_n(q)$ has the prescribed prime factorisation.
\Magma~\cite{BCP} code generating the data for Tables~1--5 mentioned
in Theorems~\ref{the:one} and~\ref{the:two} is available at~\cite{G}. 
\ep\mb

\section{The tables}\label{sec:tables}

By Lemma~\ref{lemma:phiq}\eqref{lemma:phiq1modn} the prime factorisation
of $\Phi_n^*(q)$ has the form $\prod_{i\ge1}(in+1)^{m_i}$ where $m_i=0$ if
$in+1$ is not a prime. It is convenient to encode this prime factorisation as
$\Phi_n^*(q)=\prod_{i\in I}(in+1)$ where $I$ is a multiset,
and for each $i\in I$ the prime divisor $in+1$ of $\Phi_n^*(q)$ is
repeated $m_i$ times in $I=I(n,q)$.
For example, $\Phi^*_4(8)=65=(4+1)(3\cdot 4+1)$ so $I(4,8)=\multiset{1,3}$
and $\Phi^*_5(3)=121=(2\cdot 5+1)^2$ so $I(5,3)=\multiset{2,2}$.
To save space, we omit the double braces in our tables and denote the empty
multiset (corresponding to $\Phi^*_6(2)=1$) by `$-$'.
All of our data did not conveniently fit into Table~\ref{tab:onea}, so
we created subsidiary tables \ref{tab:q=2}, \ref{tab:six}, \ref{tab:oneb}
for $n=2$, $n=6$ and $n\ge19$, respectively.
For $n$ and $q$ such that $\Phi_n^*(q)\le n^4$
Tables~\ref{tab:onea} and  \ref{tab:oneb}  
record in row $n$ and column $q$ the multiset $I(n,q).$
The tables are the output from Algorithm~\ref{alg:Mstarnbig}
with  $c=1$ and $k = 4$.

Table~\ref{tab:two} exhibits data for two different theorems.
For Theorem~\ref{the:two} we record the triples $(n,q,I)$
for which $n\ge3$ and $\Phi_n^*(q)$ has prime
factorisation $\prod_{i \in I}(i\,n+1)$ where
$I\subseteq\multiset{1, 1, 1, 2, 3, 4}.$ 
For Theorem~\ref{the:Ac} we also list the possible degrees $c$ that can arise,
namely $c_0\le c\le c_1$. 

\begin{center}
\begin{table}[!ht]
{\small
\begin{tabular}{c|ccccccccccc}
\toprule
$n\backslash q$ & $2$ & $3$ & $4$ & $5$ & $7$ & $8$ & $9$& $11$ & $13$ & $17$ & $19$ \\ 
\midrule
2 & \textup{Table~\ref{tab:q=2}}&&&&&&&&\\
3 &  2& 4& 2& 10&6& 24  & & &  20\\
4 & 1&1&4&3& 1,\;1& 1,\;3 &  10 &  15&1,\;4&1,\;7&45\\
5 & 6&2,\; 2&2,\;6&\\
\hline
6 & \textup{Table~\ref{tab:six}}&&&&&&&&\\
7 & 18&156&\\
8 & 2&5&32&39&150& &2,\;24\\
9 & 8& 84& 2,\;8&\\
\hline
10 &1& 6&4&52&1,\;19&1,\;33 &118\\
11 & 2,\;8&\\
12 & 1& 6&20&50& 1,\;15&3,\;9 & 540& 1,\;93\\
13 & 630&\\
\hline
14 & 3&39&2,\;8& 2,\;32&\\
15 &  10&304&10,\;22&\\
16 &  16&1,\;12&\\
18 & 1& 1,\;2&2,\;6&287& & 4845&\\
$\ge19$ & \textup{Table~\ref{tab:oneb}}&&&&&&&&\\
\bottomrule
\end{tabular}
\caption{Triples $(n,q,I)$ with $\Phi^*_n(q)\leq n^4$ and prime factorisation~$\Phi^*_n(q)=\prod_{i\in I} (in+1) .$}\label{tab:onea}
}
\end{table}
\end{center}

\begin{center}
\def\kk#1{\kern-2.5pt{#1}\kern-2.5pt}
\begin{table}[!ht]
{\small
\begin{tabular}{l|cccccccccccccccc}
\toprule
$q$ & 2 &  3 & $2^2$ & 5 & 7 & $2^3$ & $3^2$& 11 & 13 & 17 & 19 &23&$5^2$&$3^3$&29&31\\ \hline
 $q$ & \kk{43}&\kk{47}&\kk{59}&\kk{71}&\kk{79}&\kk{103}&\kk{127}&\kk{191}&\kk{223}&\kk{239}&\kk{383}&\kk{479}&\kk{1151}&\kk{1279}&\kk{1663}&\kk{3583}\\
\bottomrule
\end{tabular}
\caption{Prime powers $q\le 5000$ with $\Phi^*_2(q)\le 2^4=16$,
 see Remark~\ref{rem:one}.}\label{tab:q=2}
}
\end{table}
\end{center}

\begin{center}
\begin{table}[!ht]
{\small
\begin{tabular}{l|ccccccccccc}
\toprule
$q$ & $2$ &  $3$ & $4$ & $5$ & $7$ & $8$ & $9$& $11$ & $13$ & $16$ & $17$\\ 
$I$ &  $-$ & 1 & 2 & 1 & 7 & 3 & 12 & 6 & 26 & 40 & 1,\,2 \\
\hline
$q$&   $19$ &$23$& $25$ & $27$ & $29$ & $31$ & $32$& $41$ & $47$ &
 $53$ & $59$ \\
$I$& 1,\,1,\,1 & 2,\,2 & 100 & 3,\,6 & 45 & 1,\,1,\,3& 55 & 91 & 1,\,17 &
153 & 1,\,27 \\
\bottomrule
\end{tabular}
\caption{Pairs $(q,I)$ with $\Phi^*_6(q)\leq 6^4$ and prime factorisation~$\Phi^*_6(q)=\prod_{i\in I} (in+1)$ where $-$ means $\multiset{\,}$.}\label{tab:six}
}
\end{table}
\end{center}

\begin{center}
\begin{table}[!ht]
{\small
\begin{tabular}{c|cccc||c|cc||c|cc}
\toprule
  $n\backslash q$ & $2$ & $3$ & $4$ & $5$& $n\backslash q$ & 2& $3$ & $n\backslash q$ & $2$ \\
\hline
20 & 2&59&3084& &33 &  18166&   & 50 &  5,\;81\\
21 &  16& & & &34 &  1285& &  54 &  1615\\
22 &  31& 3,\;30&& & 36 & 1,\;3&14742& 60 &  1,\;22\\
\hline
24 &  10& 270&4,\;28&& 38 &  4599& & 66 &  1,\;316\\
26 & 105& 15330& & &40 &  1542& & 72 &  6,\;538\\
27 & 9728& & & &42 &  129& 1,\;54& 78 &  286755\\
\hline
28 & 1,\;4&1,\;589& & &44 &  9,\;48& & 84 &  17,\;172\\
30 & 11&1,\;9& 2,\;44& 2,\;254 &  46 &  60787& & 90 &  209300\\
32 &  2048& &  & &48 &  2,\;14&&\\
\bottomrule
\end{tabular}
\caption{All $(n,q,I)$ with $n\ge19$, $\Phi^*_n(q)\leq n^4$, and~factorisation~$\Phi^*_n(q)=\prod_{i\in I} (in+1).$}\label{tab:oneb}
}
\end{table}
\end{center}

\begin{center}
\begin{table}[!ht]
{\small
\begin{tabular}{c|c|c|cc||c|c|c|cc||c|c|c|cc}
\toprule
 $\,n\,$ & $\,q\,\,$ & $\,\,I\,\,$ & $\,c_0$ & $c_1\,$ & 
 $\,n\,$ & $\,q\,\,$ & $\,\,I\,\,$ & $\,c_0$ & $c_1\,$ & 
 $\,n\,$ & $\,q\,\,$ & $\,\,I\,\,$ & $\,c_0$ & $c_1\,$ \\
\hline
3&2&2& & &   4&13&1,\,4&17&17&  8&2&2&17&34\\ \cline{11-15}
3&3&4& & &   4&47&1,\,3,\,4&17&17&  10&2&1&15&42\\ \cline{6-10}
3&4&2& & &   6&2&$-$&15&26& 10&4&4&41&42\\ \cline{11-15}
3&9&2,\;4& & &  6&3&1&15&25&    12&2&1&15&50\\ \cline{11-15}
3&16&2,\;4& & &  6&4&2&15&26&  14&2&3&43&58\\ \cline{1-5}\cline{11-15}
4&2&1&15&18&  6&5&1&15&25&   18&2&1&19&74\\
4&3&1&15&18&  6&8&3&19&26&   18&3&1,\;2&37&73\\ \cline{11-15}
4&4&4&17&18&  6&17&1,\;2&15&25&  20&2&2&41&82\\ \cline{11-15}
4&5&3&15&17&  6&19&1,\,1,\,1&15&25&  28&2&1,\;4&113&114\\ \cline{11-15}
4&7&1,\;1&15&17&  6&31&1,\,1,\,3&19&25& 36&2&1,\;3&109&146\\ \cline{11-15}
4&8&1,\;3&15&18&  &&&&&  &&&&\\
\bottomrule
\end{tabular}
\caption{
For Theorem~\ref{the:two} we list all $(n,q,I)$ where $n\ge3$
and $\Phi_n^*(q)$ has prime factorisation
$\prod_{i\in I} (in+1)$ with $I\subseteq\multiset{1,1,1,2,3,4}$. For
Theorem~\ref{the:Ac} we also list the possible degrees $c$ where
$c_0\le c\le c_1$ and in this case we must have $n\ge4$. Here $-$
denotes the empty multiset.}\label{tab:two}  
}
\end{table}
\end{center}

\section{An Application}\label{sec:Ac}

Various studies of configurations in finite projective spaces have
involved a subgroup $G$ of a projective group $\PGL(d,q)$ (or
equivalently, a subgroup of $\GL(d,q)$) with order divisible by
$\Phi^*_n(q)$ for certain $n, q$.  This situation was analysed in
detail by Bamberg and Penttila~\cite{BambergPenttila} for the cases
where $n > d/2$, making use of the classification in \cite{GPPS}. In
turn, Bamberg and Penttila applied their analysis to certain
geometrical questions, in particular proving a conjecture of Cameron
and Liebler from 1982 about irreducible subgroups with equally many
orbits on points and lines \cite[Section~8]{BambergPenttila}.
In their group theoretic analysis Hering's theorem
\cite[Theorem~3.9]{Hering} was used repeatedly, notably to deal
with the `nearly simple cases' where $G$ has a normal subgroup $H$
containing $\mathrm{Z}(G)$ such that $H$ is absolutely irreducible,
$H/\mathrm{Z}(G)$ is a nonabelian simple group, and
$G/\mathrm{Z}(G)\leq\mathrm{Aut}(H/\mathrm{Z}(G))$.
Incidentally, the missing cases $(n,q)=(2,17)$ and $(2,71)$ mentioned
in Remark~\ref{rem:one} do not
affect the conclusions in~\cite{BambergPenttila}.

To  study other  related geometric  questions we  have  needed similar
results which  allow the parameter  $n$ to be  as small as $d/4$.  We give
here an example of how our extension of Hering's results might be used
to deal with  nearly simple groups in this more  general case where no
existing  general classifications  are applicable.  
For example, there are several theorems about translation planes that include
restrictive hypotheses such as two-transitivity~\cite{BJJMa, BJJMb, BJJMc}.
In order to remove some of these restrictions, we require results similar
to Theorem~\ref{the:Ac} for all nearly simple groups.
For  simplicity we
now consider representations of the alternating or symmetric groups
of degree $c\geq 15$ with $\Phi^*_n(q)\mid c!$
and, as we see below, $c-1\ge n\ge (c-2)/4$.

\begin{theorem}\label{the:Ac}
Let $G\le\GL(d,q)$ where $G\cong\Alt(c),\Sym(c)$, for some  $c\geq 15$,
and suppose that $\Alt(c)$  acts  absolutely irreducibly on $(\F_q)^d$
where $q$ is a power of the prime~$p$. Suppose
$\Phi^*_n(q)$ divides $c!$ for some $n\geq d/4$.
Then $n\ge4$, $d=c-\delta(c,q)$ where $\delta(c,q)$ equals $1$
if $p\nmid c$,  and $2$ if~$p\mid c$, also $c_0 \le c \le c_1$, 
and $\Phi^*_n(q)$ has prime factorisation $\prod_{i\in I}(in+1)$, where all
possible
values for $(n,q,I, c_0, c_1)$ are listed in \emph{Table~$\ref{tab:two}$}.
\end{theorem}

\begin{proof}
The smallest and the second smallest dimensions for $\Alt(c)$ and
$\Sym(c)$ modules over $\F_q$ are very roughly, $c$ and $c^2/2$ respectively.
The precise statement below follows from
James~\cite[Theorem~7]{J}, where the dimension formula $(\ast)$ on~p.\;420
of~\cite{J} is used for part~(ii).
Since $c\geq  15$, these results show that either:
\begin{enumerate}[{\rm (i)}]
  \item $(\F_q)^d$ is the fully deleted permutation module for $\Alt(c)$ with
    $d=c-\delta(c,q)$, or
  \item $d\geq c(c-5)/2$.
\end{enumerate}

In particular, since $c\ge15$ and $n\ge d/4$, we have $n\ge 4$.
Since $n>2$ it follows from Theorem~3.23 of~\cite{CW} that
$\Phi_n^*(q) > 1$ except
when $n=6$ and $q=2$. As the case $(n,q)=(6,2)$ is included in
Table~\ref{tab:two}, we assume henceforth that $\Phi_n^*(q) > 1$. Thus
$\Phi^*_n(q)=r_1^{m_1}\cdots r_\ell^{m_\ell}$
where $\ell\ge1$, each $r_i$ is a prime, and each $m_i\ge1$. Then
$r_i=a_i n+1$ for some $a_i\ge1$ by
Lemma~\ref{lemma:phiq}\eqref{lemma:phiq1modn},
and since $r_i$ divides $|\mathrm{S}_c|=c!$ we see $c\ge r_i$.
Let $r$ be the largest prime divisor of $\Phi^*_n(q)$,
so $c\geq r\geq n+1 > d/4$. In
case (ii) this   implies that $c > c(c-5)/8$ which contradicts the assumption
$c\geq15$. Thus case (i) holds. 

The inequalities $c-2\le d$ and 
$d\le 4n$ show $a_i n+1\le c\le 4n+2$ and hence $a_i\le 4$.
The exponent $m_i$ of $r_i$ is severely constrained. If $a_i\ge2$, then
\[
  r_i=a_i n+1\ge 2n+1\ge \frac {d+3}2\ge\frac{c+1}2>\frac c2.
\]
Thus the prime~$r_i$ divides
$c!$ exactly once, and $m_i=1$. If $a_i=1$, then a similar argument shows 
$r_i=n+1\ge \frac {d+4}{4}\ge\frac{c+2}{4}\ge \frac{17}{4} >4$. The
inequalities $r_i>\frac{c}{4}$ 
and $r_i>4$ imply that $r_i$ divides $c!$ at most three times, and $m_i\le3$.
In summary, $\Phi^*_n(q)$ divides $f(n):=(n+1)^3(2n+1)(3n+1)(4n+1)$.
Since $n\ge4$, we have $f(n)\le 16n^7$.
All possible pairs $(n,q)$ for which $\Phi^*_n(q)\mid f(n)$ can be computed
using Algorithm~\ref{alg:Mstarnbig} with input $c=16$, $k=7$. The output
is listed in Table~\ref{tab:two}, and computed using~\cite{G}.

For given $n$ and $q$ the possible values for $c$ form an interval
$c_0\le c\le c_1$. Since $c-\delta(c,q)=d\le 4n$
the entries $c_0, c_1$ in Table~\ref{tab:two} can be determined as follows:
$c_0=\max(r,15)$ where~$r$ is the largest prime divisor of $\Phi_n^*(q)$,
and $c_1=4n+\delta(4n+2,q)$. $\hfill\ensuremath{\square}$
\end{proof}

\subsection*{Acknowledgements}

We thank the referee for numerous helpful suggestions.
The first, third and fourth authors acknowledge the support of  
ARC Discovery grant DP140100416.

\end{document}